\documentclass[a4paper]{article}
\usepackage{amsmath}
\usepackage{amssymb}
\usepackage[matrix,arrow,tips,curve,ps, all]{xy}
\usepackage{theorem}
\usepackage{calc}
\usepackage{mathrsfs}
\usepackage{graphicx}
\usepackage{paralist}
\usepackage{comment}
\usepackage{enumerate}

\newtheorem{theo}{Theorem}[section]

\newtheorem{prop}[theo]{Proposition}

\newtheorem{lem}[theo]{Lemma}

{\theorembodyfont{\normalfont}
\newtheorem{rem}[theo]{Remark}

}

\newenvironment{proof}[1][Proof]{\noindent \textbf{#1.~}}
{\hfill $\Box$} 

\makeatletter
\@addtoreset{equation}{section}

\makeatother

\newcommand{\C}{\mathbf{C}}

\renewcommand{\P}{\mathbf{P}}

\renewcommand{\O}{\mathcal{O}}
\renewcommand{\H}{\mathrm{H}}
\newcommand{\h}{{h}}

\newcommand{\Sym}{\mathop{\mathrm{Sym}}\nolimits}

\newcommand{\PGL}{\mathop{\mathrm{PGL}}\nolimits}

\newcommand{\zero}{^{\circ}}

\renewcommand{\mod}{\mathscr}
\newcommand{\hilb}{\mathcal}

\renewcommand{\epsilon}{\varepsilon}

\makeatletter
\renewcommand\paragraph{\@startsection{paragraph}{4}{\z@}%
                                    {-3.25ex \@plus -1ex \@minus -.2ex}%
                                    {6pt \@plus 2pt \@minus .4pt}%
                                    {\normalfont\normalsize\bfseries}}
\makeatother

\makeatletter
\renewcommand\subparagraph{\@startsection{subparagraph}{3}{\z@}%
{-9pt \@plus 3pt \@minus .6pt}%
{3pt \@plus 1pt \@minus .2pt}%
{\normalfont\normalsize}%
}
\makeatother

\makeatletter
\renewcommand\section{\@startsection {section}{1}{\z@}%
                                   {-3.2ex \@plus -1ex \@minus -.2ex}%
                                   {2.3ex \@plus.2ex}%
                                   {\normalfont\Large\bfseries}}
\makeatother

\title{On universal Severi varieties of low genus $K3$ surfaces}

\author{Ciro Ciliberto and Thomas Dedieu}

\makeatletter
\renewcommand\@maketitle{%
  \newpage
  \null
  \vskip 2em%
  \begin{center}%
  \let \footnote \thanks
    {\Large \textbf \@title \par}%
    \vskip 1em%
    {\large
      \lineskip .5em%
      \begin{tabular}[t]{c}%
        \@author
      \end{tabular}\par}%
    \vskip 0.5em%
  \end{center}%
  \par
  \vskip 1.5em}
\makeatother

\begin{document}

\maketitle

\begin{abstract}
We prove the irreducibility of  universal Severi varieties
parametrizing irreducible, reduced, nodal hyperplane sections of
primitive $K3$ surfaces of genus $g$, with $3\le g \le 11$,
$g\neq 10$. 
\end{abstract}

\section*{Introduction}

F. Severi was one of the first algebraic geometers who
stressed the importance, for moduli and enumerative problems, 
of studying  the families  $V_{d,g}$ of irreducible, nodal, plane
curves of degree $d$ and geometric genus $g$
(see e.g.  \cite [Anhang F] {Sev}). 
In particular he proved that the varieties $V_{d,g}$ all have the
\emph{expected dimension} $3d+g-1$ (equal to the dimension of the
linear system  of all curves of degree $d$ minus the number of nodes)
and asserted, but did not succeed to prove,  that
they are irreducible, a result due to J. Harris in 
\cite{harris}.  For this reason,  the $V_{d,g}$'s have been called
 \emph{Severi varieties}.

The notion of Severi variety can be extended to families of nodal
curves on any surface, and 
analogous irreducibility problems naturally arise. These are in general
hard questions, even for rational surfaces.
For instance,  irreducibility is known to hold for Hirzebruch surfaces
\cite{tyomkin} and for rational curves on Del Pezzo's surfaces
\cite{testa},  with one notable (and understood) exception 
for  Del Pezzo's surfaces of degree 1. 
On the other hand, for surfaces with positive canonical bundle, 
Severi varieties have in general a quite unpredictable behaviour:
examples are given in \cite{chiantini-ciliberto} of surfaces with
reducible Severi varieties, and even with components of Severi
varieties of dimension different from the expected one.

In this note we concentrate on Severi varieties on $K3$ surfaces, which,
as in the planar case, are quite interesting in relation with modular
and enumerative problems.
Here one can consider \emph{universal} Severi varieties parametrizing
irreducible, reduced, nodal curves on primitive
$K3$ surfaces, i.e. those polarised  by an indivisible,  ample line
bundle of  genus $g$. 
Conjecturally, all these varieties should be irreducible 
(see \S \ref{ssec:USV}). This however seems, at the moment, to be
quite hard to prove. 
Our result is Theorem \ref {theo:main}  asserting the irreducibility
in the range $3\le g \le 11$, $g\neq 10$.  This partially
affirmatively answers  a question posed 
by the second author in \cite {dedieu} and it is, as far as we know,
the first irreducibility 
result for Severi varieties of $K3$ surfaces.
In a nutshell, the proof relies on two facts. 
First, in the asserted range, 
 the hyperplane sections  with a given number
 $\delta$ of nodes of  the $K3$'s in question, embedded in $\P^ g$
 as surfaces of degree $2g-2$,
fill up the whole component of nodal degenerate canonical curves with
$\delta$ nodes in the Hilbert scheme
(see Proposition \ref {prop:fibres} or \cite {fkps}).  Secondly, using a 
degeneration technique due to Pinkham \cite {pinkham}, we prove that
all components of a certain flag Hilbert scheme pass through
some \emph{cone points}, where, on the other hand, we are able to prove
smoothness of the flag Hilbert scheme, which is then irreducible at
those points. 
Both ideas are inspired by \cite{ clm93, cm90}. 

In \S \ref {sec:K3} we recall general facts about $K3$ surfaces (see
\S \ref {ssec:gen}); some basics   
 about Severi varieties on them, like existence and dimensions (see
\S \ref {ssec:severi}); about universal Severi varieties, 
 recently considered in  \cite {fkps} (see \S \ref {ssec:USV}), and
 related moduli problems (see \S  
 \ref {ssec:moduli}). Statement and proof of Theorem \ref {theo:main}
 are in \S \ref {sec:irreducibility}. 
We take a Hilbert schematic viewpoint, which
 we set up in  \S  \ref {sec:setting}. We recall then
 Pinkham's technique of degeneration to cones in \S \ref {ssec:defo},
 and the use of graph curves to compute cohomology of normal bundles
 in \S \ref {ssec:graph}. Applying  this machinery,  the proof,
 presented in \S \ref {ssec:proof}, turns out to be quite simple.

\smallskip
{\bf Acknowledgements}
The second author wishes to thank the Groupement de Recherche europ\'een
Italo-Fran\c cais en G\'eom\'etrie Alg\'ebrique (CNRS and INdAM) 
for funding his stay at the university of Roma Tor Vergata during part
of the preparation of this work.

\section{$K3$ surfaces and their Severi varieties}\label{sec:K3}

\subsection{Generalities}\label {ssec:gen}

A \emph{$K3$ surface} $X$ is a smooth  complex
projective surface with $\Omega_X^2\cong \mathcal O_X$
and $h^1(X,\mathcal O_X)=0$.
A \emph{primitive $K3$ surface of genus $g$} is a pair $(X,L)$, where
$X$ is a $K3$ surface, and $L$ is an indivisible, nef line bundle on
$X$, such that $|L|$ is without fixed component and $L^2=2g-2$ (hence
$g\ge 1$).
Given such a pair, $|L|$ is base point free, and the morphism
$\varphi_{|L|}$ determined by this linear system is
birational if and only if $L^2>0$ and $|L|$ does not contain any
hyper\-el\-liptic curve (hence $g \ge 3$).
In the latter case, the image of $\varphi_{|L|}$ is a surface of
degree $2g-2$ in $\P^g$, with canonical singularities, and whose
general hyperplane section is a \emph{canonical curve} of genus $g$
(see \cite{saint-donat}).

For all $g\ge 2$, we can consider the \emph{moduli stack}
$\mod{B}_g$ of primitive $K3$ surfaces of genus $g$, which is smooth,
of dimension 19 (see \cite {BHPV04,palaiseau}). 
For  $(X,L)$ very general in $\mod{B}_g$, the Picard group of $X$ is
generated by the class of $L$, and $L$ is very ample if $g\ge 3$.

\subsection{Severi varieties}\label{ssec:severi}
Given a $K3$ surface  $(X,L)$ of genus $g$ and two integers $k$
and $h$,  consider
\begin{equation*}
V_{k,h}(X,L):=\left\{ C\in\vert kL\vert\ \mbox{ irreducible and
    nodal with}\ 
  g(C)=h \right\},
\end{equation*}
where $g(C)$ is the \emph{geometric genus} of $C$, i.e. the genus of its normalization, so that
$C$ has $g-h$ nodes. 
$V_{k,h}(X,L)$, called the
\emph{$(k,h)$--Severi variety} of $(X,L)$ (or simply \emph{Severi
  variety} if there is no danger of confusion),  is a functorially
defined, locally closed subscheme of the projective space $\vert
kL\vert$ of dimension $1+k^2(g-1)=:p_a(k)$, which is the arithmetic
genus 
of the curves in $\vert kL\vert$.
We will drop the index $k$ if $k=1$ and we 
may drop the indication of the pair $(X,L)$ if there is no danger of
confusion. 

\begin{theo}
\label{theo:exp-descr} Let 
$k\ge 1$ and $0\le h \le p_a(k)$.
The variety $V_{k,h}$, if not empty, is smooth of dimension $h$.
If $(X,L)$ is general in $\mod{B}_g$, then $V_{k,h}$ is not empty.
\end{theo}

The first assertion is classical and standard in deformation theory 
(see \cite{Sev} and, more recently, e.g. \cite{chiantini-ciliberto,
  dedieu, Tan}).  
The second part is a consequence of the main theorem in \cite{chen}
(see also Mumford's theorem in \cite [pp. 365--367] {BHPV04}).

If $(X,L)$ is general, $V_{k,0}$ is reducible,  consisting of a finite
number of points (for the degree of $V_{k,0}$,  see
\cite {beauville-K3, yau--zaslow}). One might instead expect that 
if $(X,L)$ is general and $h\ge 1$, then $V_{k,h}$ is irreducible. 
This is trivially true for $h=p_a(k)$ and not difficult for 
$h=p_a(k)-1$ (the reader may easily figure out why), but is 
complicated as soon as $h$ gets lower. 
This conjecture, if true, certainly
will not be easy to prove. As a first approximation, 
one may propose a weaker irreducibility conjecture concerning
\emph{universal Severi varieties} (see \cite {dedieu}), which we now
recall.

\subsection{Universal Severi varieties}\label{ssec:USV}

For any $g\ge 2$,
$k\ge 1$ and $0\le h \le p_a(k)$, one can consider 
a stack $\mod{V}^g_{k,h}$ (see \cite [Proposition 4.8] {fkps}), called
the  \emph{universal Severi variety}, which is pure and smooth of
dimension $19+h$, 
endowed with a morphism $\phi^g_{k,h}: \mod{V}^g_{k,h}\to
\mod{B}_g\zero$, where $\mod B_g\zero$ is a suitable dense open
substack of $\mod B_g$. 
The morphism $\phi^g_{k,h}$ is smooth on all components of
$\mod{V}^g_{k,h}$, and its fibres are described in the following
diagram:
\[
\xymatrix@C=10pt{
\mod{V}^g_{k,h} \ar@{}[r]|(0.4){\supset} \ar[d]_{\phi^g_{k,h}}
& V_{k,h}(X,L) \ar[d] \\  
\mod{B}_g\zero \ar@{}[r]|(0.4){\ni} & (X,L)
}
\]
Thus a point of $\mod{V}^g_{k,h}$ can be regarded as a pair $(X,C)$
with $(X,L)\in \mod B_g$ and $C\in V_{k,h}(X,L)$.

One can conjecture that \emph{all universal Severi varieties
$\mod{V}^g_{k,h}$ are irreducible}. 
This does not imply the irreducibility of the \emph{pointwise
Severi varieties} $V_{k,h}(X,L)$, even if $(X,L)$ is general in
$\mod{B}_g$. 
The conjecture rather means that the monodromy of the
morphism $\phi^g_{k,h}$ transitively permutes the components
of the fibre $V_{k,h}(X,L)$, for $(X,L)\in\mod{B}_g$ general. 
This makes sense even if $h=0$, when the pointwise Severi variety
$V_{k,0}(X,L)$ is certainly reducible.

In addition to its intrinsic
interest, this conjecture is motivated by the results in
\cite{dedieu}, where it is shown that (a weak
version of) it implies the non--existence of rational
map $f:X \dashrightarrow X$ with $\deg (f)>1$ for a general
$K3$ surface $(X,L)$ of a given genus $g$. 
Very recently a proof of this result, based on quite delicate degeneration 
argument,  has
been proposed by Xi Chen \cite{chen2}.

\subsection {The moduli map}\label{ssec:moduli}

There is a natural \emph {moduli map} 
$\mu^g_{k,h}: \mod{V}^g_{k,h}\to \mod M_h$, where $\mod M_h$ is the
\emph{moduli stack} of  
curves  of genus $h$. 
The case $k=1$, $h=g$ has been much studied. It is related to 
the behaviour of the \emph{Wahl map} $w_C:\bigwedge^2 \H^0\left(C,
  \omega_C\right) \to \H^0\left(C, \omega_C^3\right)$ of a smooth
curve $C$ of genus $g$,  to extension properties of canonical curves
and to the classification of \emph{Fano varieties of the principal
  series} and of \emph{Mukai varieties}. We will not dwell 
recalling  all results on this subject, deferring the reader to the 
current literature (see, in chronological order, 
\cite  {mori-mukai, wahl,beauville-merindol, mukai1, cm90,
cukierman-ulmer,mukai2, clm93, clm98, beauville-fano}). Only recently
the nodal case $h<g$, $k=1$, received the deserved attention. We
recall the following theorem. 

\begin{theo}
\label{theo:fkps}
Assume $3\le g \le 11$ and $0\le h\le g$. For
any irreducible component $\mod{V}$ of $\mod{V}^g_{h}$,  the moduli map
${\mu^g_h}_{\vert \mod{V}}: \mod{V} \to \mod{M}_h $
is dominant, unless $g=h=10$.
\end{theo}

The case $h=g$ is in the series of papers
\cite{mori-mukai,mukai1,mukai2} (see also \cite{beauville-fano}).
The rest is in \cite {fkps}.

\begin{rem}
As stated in \cite{fkps}, the theorem applies only
for $h\ge 2$. The case $h=0$ is trivially true. 
The proof in \cite{fkps} applies to the case  $h=1$  if 
$3\le g \le 11$ and $g\neq 10$ as well.
The case $h=1, g=10$ is not covered
by the original argument 
(see also \cite [last lines of the proof of Theorem 5.5]{fkps}),
but can also be fixed. We do not dwell on this here. 
 \end{rem}

In the recent paper \cite{halic}, the moduli map $\mu^g_{k,h}$ has
been studied also for $g\ge 13$, any $k$ and $h$ sufficiently large
with respect to $g$, proving that, 
as one may expect, $\mu^g_{k,h}$  is generically finite to its image
in these cases. The remaining cases for $g,h,k$ are very interesting
and  still widely open.

\section{The main theorem}
\label{sec:irreducibility}

The aim of this paper is to prove the following result, 
which affirmatively answers the conjecture in \S\ref {ssec:USV} in
some cases.

\begin{theo}
\label{theo:main}
For $3\le g \le 11$, $g\neq 10$ and $0\le h
\le g$, the universal Severi variety $\mod{V}^g_{h}$
is irreducible.
\end{theo}

By adopting a Hilbert schematic viewpoint and inspired by 
\cite{cm90}, we find a flag Hilbert scheme $\mod F_{g,h}$,
with a rational map $\mod F_{g,h}\dasharrow \mod{V}^g_{h}$
dominating  all components of $\mod{V}^g_h$, and we
prove that $\mod F_{g,h}$ is irreducible
(see Theorem \ref{theo:hilb}). To show this, we exhibit smooth points
of $\mod F_{g,h}$ which are contained in all irreducible components of 
$\mod F_{g,h}$ (see \S \ref {ssec:proof}).

\subsection{The Hilbert schematic viewpoint}
\label{sec:setting}

For any $g\ge 3$, we let $\hilb{B}_g$ be the component of the
Hilbert scheme of surfaces in $\P^g$ whose general point parametrizes
a primitive $K3$ 
surface of genus $g$. An open subset of $\hilb{B}_g$ is a $\PGL(g+1,
\C)$-bundle over the open subset of $\mod{B}_g$ corresponding to pairs
$(X,L)$ with very ample $L$.
The variety $\hilb{B}_g$ is therefore irreducible of  dimension
$g^2+2g+19$.

Let $\hilb{C}_g$ be the component of the Hilbert scheme of curves in
$\P^g$ whose general point parametrizes a \emph{degenerate} canonical
curve of genus $g$, 
i.e. a smooth canonical curve of genus $g$  lying in a hyperplane
of $\P^ g$. 
An open subset of $\hilb{C}_g$ is a
$|\O_{\P^g}(1)|\times\PGL(g,\C)$--bundle over the open subset of
$\mod{M}_g$  parametrizing non--hyperelliptic curves, so 
$\hilb{C}_g$ is irreducible of dimension 
$g^2+4g-4$.

Let $\hilb{F}_g$ be the component of the \emph{flag Hilbert scheme} of
$\P^g$ (see \cite{kleppe, Sernesi-book})
whose general point is a pair $(X,C)$ with $X \in \hilb{B}_g$ general
and $C\in \hilb{C}_g$ a general hyperplane section of $X$. An open
subset of $\hilb{F}_g$ is a $\P^g$--bundle over an open subset of
$\hilb{B}_g$. As such it is irreducible of dimension 
$g^2+3g+19$.

Let $0 \le h \le g$. We denote by $\hilb{C}_{g,h}$ the
Zariski closure of the locally closed, functorially defined, 
subset of $\hilb{C}_g$ formed by
irreducible, nodal, genus $h$ curves. It comes with a moduli map 
$
c_{g,h}: \hilb{C}_{g,h} \dasharrow \mod{M}_h$, which is dominant.
Up to projective transformations, the fibre
over a curve $C \in \mod{M}_h$ is a dense open subset of 
$\Sym^{\delta}\left(\Sym^2 (C)\right)$, with $\delta = g-h$, and is
therefore irreducible. So $\hilb{C}_{g,h}$ is irreducible,
of dimension  $g^2+4g-4-\delta$.

We let $\hilb{F}_{g,h}$ be the inverse image of $\hilb{C}_{g,h}$ under
the projection $\hilb{F}_g \to \hilb{C}_g$. 
We have a natural dominant map $m_{g,h}: \hilb{F}_{g,h}\dasharrow \mod
V^g_h$. Any irreducible component $\hilb F$ of $\hilb{F}_{g,h}$
dominates $\hilb{B}_g$ via the restriction of the projection
$\hilb{F}_g \to \hilb{B}_g$ (see \S\S \ref {ssec:severi} and \ref
{ssec:USV}), and 
has dimension 
\[\dim\left( \hilb{F}\right) = \dim\left( \hilb{F}_{g,h}\right) =
g^2+3g+19-\delta, \] 
where again $\delta=g-h$. We let 
$p_{g,h}: \hilb{F}_{g,h} \to \hilb{C}_{g,h}$
be the natural projection and we use the shorter notation $p_g$ for
$p_{g,g}$.

Because of the existence of the dominant map $m_{g,h}$, the following
implies Theorem \ref{theo:main}.

\begin{theo}
\label{theo:hilb}
Let  $3\le g\le
11$, $g\neq 10$, and $0\le h\le g$. Then 
$\hilb{F}_{g,h}$ is irreducible.
\end{theo}

For the proof, we need to recall a few facts, collected in the next
two subsections.

\subsection{Degenerations to cones}\label{ssec:defo}

The following lemma relies on a well known 
construction of Pinkham
\cite[(7.7)]{pinkham}, 
and is based on the fact that 
smooth $K3$ surfaces are \emph{projectively
Cohen--Macaulay}, see \cite{mayer,saint-donat}.

\begin{lem}
\label{lem:cone}
Let
$(X,C)\in \hilb{F}_{g,h}$ with $X$ a smooth $K3$ surface. 
Let $X_C$ be the cone over $C$ 
from a point $v$ in $\P^g$ off the
hyperplane in which $C$ sits.
Then one can flatly degenerate $(X,C)$
to $(X_C,C)$ inside 
the fibre $F_C$ of $p_{g,h}$ over $C$. 
\end{lem}

\begin{proof} Let $H$ be the hyperplane containing $C$.
Choose homogeneous coordinates $(x_0\!:\ldots:\!x_g)$ such that
$v=(1\!\!:\!0\!:\ldots:\!0)$ and $H$ is given by $x_0=0$.
Consider the projective transformation $\omega_t$, $t\neq 0$,
such that $\omega_t(x_0\!:\ldots:\!x_g)=(tx_0\!:x_1:\ldots:\!x_g)$.
Set $X_t=\omega_t(X)$.
Then $(X_t,C)\in F_C$ for all $t\neq 0$. Since $X$ is projectively
Cohen--Macaulay, $X_C$, with its reduced structure, is the flat limit
of $X_t$ when $t$ tends to $0$. \end{proof}\medskip

The fibre $F_C$ of $p_{g,h}$  equals the fibre of $p_g$,
 whose tangent space at the point $(X,C)$ is isomorphic to
$\H^0(X,N_{X/ \P^g}(-1))$ 
(see e.g. \cite[\S 4.5.2]{Sernesi-book}).
The next lemma computes this space at a \emph{cone point}
$(X_C,C)$ (the proof is the same as in
\cite[Theorem 5.1]{pinkham},
and relies on the fact that $C$ is projectively Cohen--Macaulay,
see \cite{rosenlicht, kleiman-martins, schreyer}).

\begin{lem}
\label{lem:grading}
Let $C$ be a reduced and irreducible, not necessarily smooth,
degenerate canonical curve in $\P^g$, of arithmetic genus $g$. 
Let $X_C$ be the cone over $C$ from a point in $\P^g$ off
the hyperplane in which $C$ sits. For all $i\ge 0$,  one has
\begin{equation}
\label{grading}
\H^0\!\left(X_C,N_{X_C/ \P^g}\right(-i)) \cong 
\bigoplus_{k\ge i} \,\H^0\!\left(C,N_{C/\P^{g-1}}(-k)\right).
\end{equation}
\end{lem}

Next we need to bound from above the dimensions of the cohomology spaces
appearing in the right--hand--side of \eqref{grading}. We use
semi--continuity, and a special type of canonical curves for which
they  can be  computed.

\subsection{Canonical graph curves}\label{ssec:graph}

A \emph{graph curve} of genus $g$ is a stable curve 
of genus $g$ consisting of $2g-2$ irreducible components of genus 0
(see \cite{bayer-eisenbud, cm90}). A graph curve has $3g-3$ nodes
(three nodes for each component), and it is determined by the dual
\emph{trivalent} graph, consisting of $2g-2$ nodes and $3g-3$
edges. If $C$ is a graph curve and its dualizing sheaf $\omega_C$
is very ample, then $C$ can be canonically embedded in $\P^{g-1}$
as a union of $2g-2$ lines, each meeting three others
at distinct points. This is a \emph{canonical graph curve}. 

\begin{prop}
\label{prop:normal-graph}
{\normalfont \cite{cm90}}
For $3\le g\le 11$, $g\neq 10$, there exists a genus $g$
canonical graph curve $\Gamma_g$ in $\P^{g-1}$, sitting in the image of 
$p_g$, such that the
dimensions of the spaces of sections of negative twists of
the normal bundle are given in the following table: 
\[ 
\begin{array}{c|*{8}{l}}
\h^0\left(N_{\Gamma_g/ \P^{g-1}}(-k)\right) \,\setminus\, g & 
3 & 4 & 5 & 6 & 7 & 8 & 9 & 11 \\
\cline{1-9} 
k=1 & 10 & 13 & 15 & 16 & 16 & 15 & 14 & 12 \\
k=2 & 6  & 5  & 3  & 1  & 0  & 0  & 0  & 0  \\
k=3 & 3  & 1  & 0  & 0  & 0  & 0  & 0  & 0  \\
k=4 & 1  & 0  & 0  & 0  & 0  & 0  & 0  & 0  \\
k \ge 5 & \multicolumn{8}{l}{0 \text{ for every } g \hfill} 
\end{array} 
\]
hence
\begin{equation}
\label{tg-fibre}
\sum\nolimits_{k \ge 1}
\,\h^0\!\left(\Gamma_g,N_{\Gamma_g/\P^{g-1}}(-k)\right) 
= 23-g.
\end{equation}
\end{prop}

\subsection{Proof of the main theorem}\label{ssec:proof}

Here we prove Theorem \ref {theo:hilb} and 
therefore also Theorem \ref {theo:main}.
The first step is the following:

\begin{prop}
\label{prop:fibres}
Let $g$ and $h$ be two integers such that $3\le g\le
11$, $g\neq 10$, and $0\le h\le g$. Let $\hilb F$ be a 
component of $\hilb{F}_{g,h}$ and let $(X,C)\in \hilb F$ be a general
point.
Then all components of the fibre $F_C$ of $p_{g,h}$ over $C$ have
dimension $23-g$, and the restriction of $p_{g,h}$ to $\hilb F$
is dominant onto $\hilb C_{g,h}$.
\end{prop}

\begin{proof}
Note that $X$ is general in $\hilb B_g$ (see \S
\ref {sec:setting}).  
As we saw,  $F_C$  equals
 the fibre of $p_g$,
  whose tangent space at $(X,C)$ is isomorphic to
$\H^0(X,N_{X/ \P^g}(-1))$. Its dimension does not depend on
the hyperplane section $C$ of $X$.  
So this is like computing the tangent space to the fibre of $p_g$ at a
general point $(X,\bar C)$ of $\hilb F_g$, with $\bar{C}$ 
general in $\hilb{C}_g$ by the case $h=g$ of Theorem \ref{theo:fkps}.

By degenerating to the cone point $(X_{\bar C},\bar C)$, by Lemma \ref
{lem:grading}, and by upper--semi--continuity, we have $h^0(X,N_{X/
  \P^g}(-1))\le \sum_{k\ge 1} h^0(\bar C,N_{\bar C/
  \P^{g-1}}(-k))$. By further degenerating 
to one of the graph curves in Proposition \ref {prop:normal-graph},
and taking into account \eqref {tg-fibre}, we have 
$h^0(X,N_{X/ \P^g}(-1))\leq 23-g$ (this argument has been extracted from
  \cite [\S 5.3] {clm93}).  So this is an upper bound for the 
dimension of $F_C$ at $(X,C)$. 
Since 
\[
23-g=\dim\left(\hilb{F}\right) - \dim\left(\hilb{C}_{g,h}\right)
\]
(see \S \ref {sec:setting}),  this equals the dimension of $F_C$ at
$(X,C)$, and the restriction 
of $p_{g,h}$ at $\hilb F$ is dominant. \end{proof}\medskip

With a similar argument  we can finish the:\medskip

\begin{proof}[Proof of Theorem \ref {theo:hilb}]
Let  $\hilb F_i$, $1\le i\le 2$, be distinct components of
$\hilb{F}_{g,h}$. 
Let $C\in \hilb C_{g,h}$ be a general point. By Proposition
\ref {prop:fibres}, there are points $(X_i,C)\in \hilb F_i$, 
and they can be assumed to be general points on two  distinct 
components $F_i$ of $F_C$, $1\le i\le 2$.
 By Lemma \ref{lem:cone}, both
$F_1$ and $F_2$  contain the cone point $(X_C,C)$. We will reach a
contradiction by showing that $(X_C,C)$ is a smooth point of $F_C$.

Since $C$ is general in $\hilb C_{g,h}$ and 
$\hilb C_{g,h}$ clearly contains the graph curves $\Gamma_g$
of Proposition \ref {prop:normal-graph}, by upper--semicontinuity 
$h^0(X_C,N_{X_C/ \P^g}(-1))$ is bounded from above by \eqref {tg-fibre}.
This proves the asserted smoothness of $F_C$ at $(X_C,C)$, concluding
the proof. \end{proof}

\begin{rem} 
\begin{inparaenum}[(i)]
\item Proposition \ref  {prop:fibres} gives a quick alternative
proof of the part $h<g$ of Theorem \ref {theo:fkps} when $g\neq 10$,
which is based on the part $h=g$.\\
\item The argument does not work for $g=10$. In fact, if $C$ is any
  curve in the image of $p_{10}$ lying on a smooth $K3$ surface,
one has 
$\h^0\left(X,N_{X/\P^g}(-1)\right)=14$,
see \cite [Lemma 1.2] {cukierman-ulmer}.
The analogue of Proposition \ref
{prop:fibres}  in this case is that all components of a general fibre of
$p_{10,h}$  have dimension $14$. 
So the image of $p_{g,h}$ has codimension 1 in $\hilb C_{g,h}$.
Luckily, and as one could expect, Theorem  \ref {theo:fkps} ensures
that the moduli map $c_{g,h}$ dominates $\mod{M}_h$ for $0\le h\le
9$.
However the argument in the final part of the proof of Theorem \ref
{theo:hilb} falls short, since we do not know whether the image of
$p_{10,h}$ is irreducible, or all of its components contain a curve
$C$ for which the fibre $F_C$ can be controlled.
\end{inparaenum}
\end{rem}

{\small
\vskip .2cm \noindent
\textsc{%
Dipartimento di Matematica, 
Universit\`a degli Studi di Roma Tor Vergata,
Via della Ricerca Scientifica,
 00133 Roma, Italy} \\
\texttt{cilibert@mat.uniroma2.it}

\vskip .4cm \noindent
\textsc{%
Universit\'e Paul Sabatier, Institut de Math\'ematiques de Toulouse,
118 route de Narbonne,
 31062 Toulouse Cedex 9, France} \\
\texttt{thomas.dedieu@math.univ-toulouse.fr}
}

\end{document}